%% file: regularityEI.tex
\documentclass[twoside,11pt]{amsart}
\usepackage[b5paper,scale=.88,centering]{geometry}  
\usepackage{mathrsfs,amssymb,mathabx}
\usepackage{enumitem}
\newlist{steps}{enumerate}{1}
\setlist[steps, 1]{label = Step \arabic*:}
\usepackage[mathlines]{lineno}
\usepackage[dvipsnames]{xcolor}
\usepackage[pagebackref,colorlinks,citecolor=Plum,urlcolor=Periwinkle,linkcolor=DarkOrchid]{hyperref}
\usepackage{graphicx}
\newenvironment{lesn}{\begin{linenomath}\begin{esn}}{\end{esn}\end{linenomath}}
 \theoremstyle{plain}
\newtheorem{theorem}{Theorem}

\newtheorem{lemma}[theorem]{Lemma}

\newtheorem{proposition}[theorem]{Proposition}

\theoremstyle{definition}
\newtheorem*{definition}{Definition}
\newtheorem*{remark}{Remark}
\title[]{Dini derivatives for exchangeable increment processes and applications}

\author{Osvaldo Angtuncio}
\address[]{}
\author{Ger\'onimo Uribe Bravo}
\address[OA and GUB]{Instituto de Matem\'aticas\\ 
Universidad Nacional Aut\'onoma de M\'exico\\
\'Area de la Investigaci\'on Cient\'ifica, Circuito Exterior, Ciudad Universitaria\\ Coyoac\'an, 04510. Ciudad de M\'exico, M\'exico }
\address[GUB]{ 
LaSoL UMI No. 2001\\ 
CNRS-CONACYT-UNAM\\ Mexico}
\input{definitions.tex}

\subjclass[2010]{
60G09
, 60G17
}
\thanks{
Research 
supported by
CoNaCyT
grant FC-2016-1946 and UNAM-DGAPA-PAPIIT grant IN115217.  
} 
\usepackage[dvipsnames]{xcolor}

\begin{document}
\begin{abstract}
Let $X$ be an exchangeable increment (EI) process whose sample paths are of infinite variation. 
We prove that, for any fixed $t$ almost surely,
\begin{esn}
\limsup_{h\to 0 \pm} \paren{X_{t+h}-X_t}/h=\infty 
\quad\text{and}\quad \liminf_{h\to 0\pm} \paren{X_{t+h}-X_t}/h=-\infty. 
\end{esn}This extends a celebrated result of Rogozin for L\'evy processes obtained in \cite{MR0242261}, 
and completes the known picture for finite-variation EI processes. 
Applications are numerous. 
For example, 
we deduce that both half-lines $(-\infty, 0)$ and $(0,\infty)$ are visited immediately for infinite variation EI processes (called upward and downward regularity). 
We also generalize the zero-one law of Millar for L\'evy processes 
by showing continuity of $X$ when it reaches its minimum in the infinite variation EI case (cf. \cite{MR0433606}); 
an analogous result for all EI processes links right and left continuity at the minimum with upward and downward regularity. 
We also consider results of Durrett, Iglehart and Miller  
on the weak convergence of conditioned Brownian bridges to the normalized Brownian excursion 
considered in \cite{MR0436353} 
and broadened to a subclass of L\'evy processes and  EI processes in \cite{MR3160578,MR3558138}. 
We prove it here for all infinite variation EI processes. 
We furthermore obtain a description of the convex minorant known for L\'evy processes 
found in  \cite{MR2978134} and extended to non-piecewise linear EI processes. 
Our main tool to study the Dini derivatives is a change of measure for EI processes which extends the Esscher transform for L\'evy processes. 
\end{abstract}
\maketitle

\section{Statement of the results}
\label{resultsSection,whichIsIntroduction}
Undoubtedly, L\'evy Processes are one of the most studied classes of stochastic processes. 
A less known class which contains them is that of Exchangeable Increment (EI) processes 
considered in general by Kallenberg in \cite{MR0394842}. 
\begin{definition}
A continuous time \cadlag\ $\re$-valued stochastic process $X=\paren{X_t, t\in [0,1]}$ 
has \emph{exchangeable increments} if for every $n\geq 1$, 
the random variables
\begin{align*}
X_{1/n},X_{2/n}-X_{1/n},\ldots, X_1-X_{(n-1)/n}
\end{align*}are exchangeable. 
\end{definition}
Clearly, all L\'evy Processes are EI since iid random variables are exchangeable. 
Therefore, one can inherit  results for L\'evy processes 
from their counterparts for EI processes, as we illustrate in this paper. 
However, conditioning a L\'evy process $X$ by its final value 
(to obtain the so called L\'evy bridges as in \cite{MR2789508} and \cite{MR3160578}) 
or  considering $(X_t-tX_1,t\leq 1)$ also yield non-L\'evy EI processes, 
so that our results can be applied more broadly.

Also, the analysis of EI processes is sometimes aided by simple combinatorial considerations. 
Indeed, for random walks, the combinatorial considerations of \cite{MR0079851} lead to a more thorough understanding of the Fluctuation Theory (study of extremes) of random walks and L\'evy processes, and in particular of the celebrated arcsine law for symmetric random walks and L\'evy processes; 
it also reobtains the following formula of \cite{MR0062867}
\begin{equation}
\label{equation:SpitzerExpectationOfMaximum}
\esp{\max_{0\leq k\leq n} X_{k/n}}=\sum_{k=1}^n \frac{1}{k}\esp{X_{k/n}^+}. 
\end{equation}
More recently, \cite{MR2825583} introduced a bijection on permutations which ultimately lead to a description of the convex minorant of a (discrete time) EI process and reinterprets the fluctuation theory of random walks. 
The Kac-Spitzer identity just displayed is interpreted as the equality in law
\begin{equation}
\label{equation:distributionalEqualityForMaximum}
\max_{0\leq k\leq n} X_{k/n}\stackrel{d}{=} \sum_{i=1}^{K_n} \bra{X_{S_i/n}-X_{S_{i-1}/n}}^+
\end{equation}where $0=S_0<S_1< \cdots< S_{K_n}=n$ is the partition obtained from a uniform stick breaking process on $\set{1,\ldots n}$ independent of $X$. 
The link with the typical fluctuation theory (of random walks and L\'evy processes) comes from considering a random $n$ independent of $X$ and geometrically distributed. 
The partition is then seen to arise from a Poisson point process and the right hand side becomes a compound Poisson distribution in the Random Walk or L\'evy process case; cf. Theorem 4 in \cite{MR2825583}. 
The description of the convex minorant for discrete time EI processes is used here to prove an analogous theorem for continuous time EI processes. 
The multidimensional case is much less studied, but   the combinatorial lemma of \cite{MR0156261}, from which one obtains the expected characteristics of the convex hull of (2D) random walks (like perimeter length or area) 
has been extended in various directions (and dimensions!) including \cite{MR3782474, MR3709116,MR3678504,MR3852455}. 
Still in the realm of  fluctuation theory, \cite{MR1232850} constructs (one-dimensional) random walks conditioned to stay positive through a bijection on permutations; this result is used here to study continuity of an EI process when it reaches its minimum. 
Away from random walks, (discrete time) EI processes of a particular type are associated to trees (with a given degree distribution) in  \cite{MR3188597} and combinatorial considerations give information on this probabilistic model. 

%

Kallenberg obtained in \cite{MR0394842} the following representation of EI processes $X$: 
there exist random variables $\alpha$, $\beta=\paren{\beta_i, i\in\na}$, and $\sigma\geq 0$ 
which are independent of an iid sequence of uniform random variables  $\paren{U_i,i\geq 1}$, 
and of a Brownian bridge $b$, 
such that
\begin{esn}
X_t=\alpha t+\sigma b_t+\sum_{i\geq 1} \beta_i\bra{\indi{U_i\leq t}-t}. 
\end{esn}When $\alpha, \beta$ and $\sigma$ are deterministic, 
the EI process $X$ is termed \emph{extremal}.
All EI processes are therefore mixtures of extremal EI processes 
and we say that $X$  has canonical parameters $(\alpha, \sigma,\beta)$. 
\begin{remark}
Our results are stated for extremal processes. 
They can be generalized by conditioning on the parameters, on the set where these satisfy the given hypotheses. 
\end{remark}

The sample paths of an extremal EI process $X$ are of infinite variation if and only if 
\begin{description}
\item[Infinite variation] either $\sigma>0$ or $\sum\abs{\beta_i}=\infty$. 
\end{description}

Our first result is the following:
\begin{theorem}
\label{RogozinForEITheorem}
Let $X$ be an extremal EI process of infinite variation. 
Then, for any fixed $t$ almost surely,
\begin{equation} 
\label{Eq:InfinitudeOfDiniDerivatives}
\limsup_{h\to 0 } \frac{X_{t+h}-X_t}{h}=\infty
\quad\text{and}\quad
\liminf_{h\to 0} \frac{X_{t+h}-X_t}{h}=-\infty
\end{equation}both from the left and from the right. 
\end{theorem}
Reversibility for EI processes (the fact that $(X_1-X_{(1-t)-},t\leq 1)$ has the same law as $X$) implies that it is enough to handle the above theorem for the right-hand derivatives. 
By exchangeability, it is enough  consider $t=0$. 
We define\begin{esn}
\imf{\overline D}{X}=\limsup_{h\to 0 +} \frac{X_{h}}{h}
\quad\text{and}\quad
\imf{\underline D}{X}=\liminf_{h\to 0+} \frac{X_{h}}{h}. 
\end{esn}

In contrast, for finite variation EI processes $X$, 
which satisfy $\sigma=0$ and $\sum\abs{\beta_i}<\infty$, 
we can write them as $X_t=\tilde \alpha t+\sum_i \beta_i\indi{U_i\leq t}$ 
where $\tilde \alpha=\alpha-\sum_i\beta_i$. 
Finite variation EI processes therefore are characterized by the parameters $(\tilde\alpha,\beta)$. 
It is well known that $\imf{\overline D}{X}=\imf{\underline D}{X}= \tilde \alpha$ almost surely (cf. \cite[Cor. 3.30]{MR2161313}). 

Theorem \ref{RogozinForEITheorem} was proved for L\'evy processes 
in \cite{MR0242261} 
by using an integro-differential equation initially found by Cram\'er 
and later recognized and analyzed as a resolvent equation by Watanabe in \cite{MR0303602}. 
Additional proofs, based on the fluctuation theory for L\'evy processes, 
may be found in \cite[Ch. 9\S47]{MR1739520} and \cite{vigon:tel-00567466}. 
Bertoin proved the limsup statement in the spectrally positive EI case when ($\beta_i\geq 0$ for all $i$) 
in \cite{MR1957515} based on the results of Fristedt from \cite{MR0288841}. 
Kallenberg takes results further by considering upper envelopes of EI processes in \cite{MR2161313} 
by clever couplings with L\'evy processes. 
These results are nevertheless insufficient to obtain Thereom \ref{RogozinForEITheorem}. 
A particular case of the above result is found in \cite[Prop. 3.5]{MR3558138} 
under an additional hypothesis on $\beta$. 
Additionally, the same proposition  proves Theorem \ref{RogozinForEITheorem} 
whenever $\sigma\neq 0$ using the law of the iterated logarithm for Brownian motion. 
Hence, we could assume that $\sigma=0$ in our proofs, but the method is robust enough to handle it. 
Actually, in the L\'evy process setting, our method can also handle general L\'evy processes 
and gives and independent proof of Rogozin's result. 
This is done in Section \ref{LevyProcessSection}, while Theorem \ref{RogozinForEITheorem} is proved 
in Section \ref{totallyAsymmetricDiniSection}.


Our next application is to show that the zero-one laws for L\'evy processes of Millar are actually valid 
(and therefore a consequence) of the following result (cf. items a and b of \cite[Thm. 3.1]{MR0433606}) 
which links behaviour when reaching the minimum with behaviour at time zero. 
\begin{definition}
An EI process $X$ is said to be \emph{upward regular} if $\inf\set{t\in[0,1]: X_t>0}=0$ almost surely. 
$X$ is \emph{downward regular} if $-X$ upward regular. 
\end{definition}
Knight has given in \cite{MR1417982} the following necessary and sufficient conditions for $X$ to admit 
a unique minimum in the extreme setting: 
\begin{description}
\item[UM] either $\sigma\neq 0$ or $\sum_{i}\indi{\beta_i\neq 0}=\infty$ or $\sum_{i}\indi{\beta_i\neq 0}<\infty$ and $\sum_{i} \beta_i\neq \alpha$. 
\end{description}
\begin{theorem}
\label{ZeroOneLawTheorem}
Let $X$ be an extremal EI process satisfying \defin{UM}. 
Let $\underline X_1=\inf_{s\in [0,1]}X_s $ and 
let $\rho$ be the unique element of $\set{t\in [0,1]: X_t\wedge X_{t-}=\underline X_1 }$. 
Then $X_{\rho}>\underline X_1$ if and only if $X$ is irregular upward 
and $X_{\rho-}>\underline X_1$ if and only if $X$ is irregular downward. 
In particular, $X$ is continuous at $\rho$ if and only if $X$ is both upward and downward regular 
and this holds on the set where $X$ has paths of infinite variation. 
\end{theorem}
Millar actually proves the above result in the L\'evy process setting at more general random times 
and refers to this as the pure behavior of L\'evy processes, 
while noting that it is rather exceptional in the class of Markov processes. 
Millar also remarks that it is this zero-one law which implies that 
the conditional law of $X_{\rho+\cdot}$ given $X_{\cdot\wedge\rho}$ 
depends only on $\underline X_1$ and $X_\rho$. 
The extension to general random times follows quite easily with Millar's arguments from the stated result above. 
When $X$ is a L\'evy process of finite-variation, necessary and sufficient conditions 
for regularity have been found by Bertoin in \cite{MR1465812} in terms of the L\'evy measure. 
We believe that a similar characterization should be available for EI processes in terms of $\beta$. 
This is left as an \defin{open problem}.

Regularity of half-lines for a L\'evy process has many other applications: 
it helps in obtaining perfectness of the zero set  
and in constructing a continuous (Markovian) local time (Theorem 6.6 of \cite{MR3155252}); 
it implies uniqueness for solutions of time-change equations used to construct multitype branching processes
(Lemma 6 in \cite{MR3689968}); 
it proves continuity of the Vervaat transform for cyclically exchangeable processes (\cite{MR3558138}); 
regularity of $(-\infty,0)$ has been used when pricing perpetual American put options, 
as a condition for smooth pasting (see the discussion on Section 1.4.4 in \cite{MR2343206}).

Our second application concerns the weak limit of and EI process $X$ ending at zero, 
conditioned on remaining above $-\eps$, as $\eps\to 0$. 
The limiting process is called the Vervaat transform of $X$, and is defined as:
\[
V_t=X_{\cdot+\rho \mod 1}
\]
\begin{theorem}
\label{DIMTheorem}
Let $X$ be an EI process with $\alpha=0$ 
which is both upward and downward regular. 
Consider $\eps>0$ and let  $X^\eps$ have the law of $X$ conditionally on $\underline X_1>-\eps$. 
Then $X^\eps\stackrel{d}{\to} V$ as $\eps\to 0$. 
\end{theorem}
Note that the above theorem always applies to infinite variation EI processes thanks to Theorem \ref{RogozinForEITheorem}. 
The above theorem was proved when $X$ is a Brownian bridge from $0$ to $0$ of length $1$ 
by Durrett, Iglehart and Miller in \cite{MR0436353}. 
The form given above is taken from  \cite{MR3558138} and is more general, 
but is actually a simple consequence of the results in that paper. 
What was lacking in the latter reference is the zero-one law at the minimum 
of our Theorem \ref{ZeroOneLawTheorem} and, 
in particular, the fact that all infinite variation EI processes reach their minimum continuously. 

Our next application is to extend the description of the convex minorant of a L\'evy process of \cite{MR2978134} 
to the EI setting. 
In the latter reference, it is noted that this description gives another interpretation 
of a fundamental fact of the fluctuation theory of L\'evy processes, namely, 
the Pecherskii-Rogozin identity of  \cite{MR0260005}. 
We will consider EI processes which do not have piecewise linear trajectories. 
By considering the extremal case, 
this happens if and only if 
\begin{description}
\item[NPL] $\sigma>0$ or $\sum_{i}\indi{\beta_i\neq 0}=\infty$. 
\end{description}We will call theses processes of the NPL type. 
The setting of infinite variation EI processes of Theorem \ref{RogozinForEITheorem} 
is an important step in the proof. 
\begin{definition}
The \emph{convex minorant} of a \cadlag\ function $\fun{f}{[0,1]}{\re}$ 
is the greatest convex function $c$ that is bounded above by $f$. 
The \emph{excursion set} is the open set
\[
\mc{O}=\set{t\in [0,1]: \imf{f}{t}>\imf{c}{t}}. 
\]Its maximal components, intervals of the form $(g,d)$, 
are termed \emph{excursion intervals} 
and they have an associated \emph{length} $d-g$,  
\emph{increment} $\imf{f}{d}-\imf{f}{g}$ ,
\emph{slope} $(\imf{f}{d}-\imf{f}{g})/(d-g)$ and 
\emph{excursion} $e(t)=\imf{f}{t+g}-\imf{c}{g+s}$ defined for $t\in [0,d-g]$. 
\end{definition}
Recall that an upper bounded family of convex functions has a convex supremum, 
which explains why the convex minorant exists. 
Let $C$ be the convex minorant of an EI process $X$ of the NPL type. 
As stated in the next result, its excursion set
\[
\mc{O}=\set{t\in [0,1]: X_t>C_t}
\]is open and of Lebesgue measure $1$. 
We will consider the following precise ordering of the excursion intervals. 
Let $\paren{V_i,i\geq 1}$ be an iid sequence of uniform random variables on $[0,1]$ 
and let $(g_1, d_1)$, $(g_2, d_2)$, $\ldots$  be the sequence of distinct excursion intervals 
which are successively discovered by the sequence $\paren{V_i}$. 
With them, we can define the sequence of lengths, slopes  and excursions $\paren{e^i}$. 
We will also consider the partition induced by a stick-breaking scheme based on $\paren{V_i}$: 
define
\[
S_0=0,\quad S_{n+1}=S_n+L_n\quad\text{and}\quad L_{n+1}=(1-S_n)V_{n+1}.
\] 
Then $\paren{L_i}$ is the uniform stick-breaking process and $S$ is the partition of $[0,1]$ 
induced by its cumulative sums. 
Note that this is a very sparse partition of $[0,1]$ which we can use to analyze $X$ by considering: 
$X_{S_i}-X_{S_{i-1}}$ and the sequence of Knight bridges 
where $K^i_t$ is the Knight transform of $X-X_{S_{i-1}}$ on $[0,L_i]$. 
The \emph{Knight transform} of an EI process $Y$ starting at zero on an interval $[0,t]$ 
and satisfying \defin{UM} 
is obtained by first defining the Knight bridge 
$K_s=Y_s-sY_t/t$, 
letting $\rho$ be the location of its (unique) minimum to finally define 
\[
s\mapsto K_{(\rho+s) \,\mathrm{mod}\, t}-K_{\rho}\wedge K_{\rho-} \text{ for }s\in [0,t]. 
\] 
%
\begin{theorem}
\label{ConMinTheorem}
Assume that the EI process $X$ satisfies NPL. 
Then, its excursion set $\mc{O}$ is open and of Lebesgue measure $1$ and 
the following equality in law holds: 
\[
\paren{ d_i-g_i, X_{d_i}-X_{g_i}, e^i}_{i\geq 1}\stackrel{d}{=}\paren{ L_i, X_{S_i}-X_{S_{i-1}}, K^i}_{i\geq 1}. 
\]
\end{theorem}

Recent papers have used the above description of the convex minorant (in the L\'evy process case) 
to develop an exact simulation method for the maximum of a stable process 
(found in \cite{2018arXiv180601870G}) and an approximate simulation method 
(albeit very efficient, cf. \cite{2018arXiv181011039G})  
for the maximum of L\'evy processes whose one-dimensional distributions can be sampled exactly. 
This is particularly relevant to Monte Carlo methods for ruin probabilities with finite and deterministic horizon. 
In \cite{MR3390137}, 
the classical Cramer-Lundberg ruin process is generalized to an exchangeable increment process on $[0,\infty)$ 
to relax the independence between claim sizes; 
these are mixtures of L\'evy processes. 
In contrast to the classical setting, 
when working under the classical net profit condition, 
the ruin probability might not converge to zero as the initial capital goes to infinity and a new net profit condition is needed. 
In the finite-horizon case, 
we would be dealing with an EI process of the type considered here; 
Theorem \ref{ConMinTheorem} would give us access to the ruin probabilities. 

We end this section with a few comments on the organization of the paper. 
Our main result is Theorem \ref{RogozinForEITheorem}; 
all others have a simple proof from Theorem \ref{RogozinForEITheorem} and more specialized results from the literature. 
A brief outline of the proof of Theorem \ref{RogozinForEITheorem}, 
which explains the organization of the paper, is as follows.

\begin{steps}
\item Assume $\alpha, \sigma=0$ and $\beta_j\geq 0$ for every $j$ (if $\beta_j\leq 0$, apply this step to $-X$). 
\begin{description}
 \item[$\bullet$ $\overline D(X)=\infty$] This follows from results of \cite{MR0288841}.
        \item[$\bullet$  $\underline D(X)=-\infty$]
 We use an exponential change of measure (which reduces to the well known Esscher transform if $X$ is a L\'evy process), 
 with parameters $\theta\in\re$ and $T\in(0,1)$, 
 to deduce that $\underline D(X)=\alpha^\theta/T+\underline D(X^\theta)$, 
 where $X^\theta$ is another EI  process, 
 $\alpha^\theta$ is a random variable 
 (not independent of $X^\theta$, although for L\'evy processes, $\alpha^\theta$ is deterministic). 
 A lower bound on the probability that an EI process with positive jumps is non-positive 
 (found in \cite{MR1739520} for L\'evy processes) implies that $\underline D(X^\theta)\leq 0$. 
 It remains to notice that the infinite variation hypothesis gives us
 $\alpha^\theta\to -\infty$ when $\theta\to -\infty$, 
 thereby proving Theorem \ref{RogozinForEITheorem} in this case. 
\end{description}
\item Assume $\sigma\neq 0$ or 
$\sum_j\abs{\beta_j}=\infty$; 
also set $\alpha=0$. 
\begin{description}
 \item[$\bullet$ $\underline D(X)=-\infty$] We note that the aforementioned lower bound is valid when $\sigma\neq 0$  
 and that it works along deterministic subsequences, 
 so that  $\liminf_{n\to\infty} X_{t_n}/t_n\leq 0$ whenever $t_n\downarrow 0$ and $X$ has only positive jumps and a Brownian component. 
 We then write $X=X^p+X^n$, where $X^p$ and $X^n$ are independent, 
 $X^n$ only has negative jumps and $X^p$ only has positive jumps and contains the Brownian component (if any). 
 If $X^p$ or $X^n$ has finite variation, we use \cite[Cor. 3.30]{MR2161313} and Step 2. 
 Hence, assume both have infinite variation. 
 We then get a random subsequence $T_n\downarrow 0$ such that $X^p_{T_n}/T_n\to -\infty$ and, using independence, we get $\liminf X^n_{T_n}/T_n\leq 0$. 
 Hence, we obtain $\underline D(X)=-\infty$. 
 \item[$\bullet$ $\overline D(X)=\infty$] Apply the previous case to $-X$.
\end{description}
\end{steps}

The paper is organized as follows: 
In Section \ref{LevyProcessSection} we present a simplified proof following the outline above in the setting of L\'evy processes. 
This is because the exponential change of measure and lower bounds on probabilities discussed above are already known. 
In Section \ref{totallyAsymmetricDiniSection}, 
we consider Theorem \ref{RogozinForEITheorem} in the case of EI processes. 
Here, we state and prove the exponential change of measure and lower bounds for probabilities. 
Finally, Section \ref{furtherApplicationSection} is devoted to the applications of our results, 
and contains the proofs of Theorems \ref{ZeroOneLawTheorem}, \ref{DIMTheorem} and \ref{ConMinTheorem}. 

\section{The L\'evy process case}
\label{LevyProcessSection}

We now illustrate the proof of Theorem \ref{RogozinForEITheorem} in the case of L\'evy processes. 
This proof is the only one published that does not use fluctuation theory for L\'evy processes and can be considered to be simpler. 
It is based on basic facts on L\'evy processes and on the Esscher transform. 
The reader might consult \cite{MR1406564} and \cite{MR1739520} for these basic facts, some of which we now recall. 
In particular, L\'evy processes satisfy the Blumenthal $0$-$1$ law and therefore the random variables 
$\imf{\overline D}{X}$ and $\imf{\underline D}{X}$ are actually constant. 

Recall that $X$ can be written as the independent sum of two L\'evy processes $X^1$ and $X^2$,  
where $X^1$ has bounded jumps and $X^2$ is compound Poisson. 
Since $\lim_{t\to 0} X^2_t/t$ exists and is finite, we see that it suffices to prove Theorem \ref{RogozinForEITheorem} when $X$ has bounded jumps. 

Assume then, that the jumps of $X$ are bounded by $1$; 
we can then determine $X$ by its Laplace transform 
\[
\esp{e^{\lambda X_t}}=e^{t \imf{\Psi}{\lambda}}
\quad\text{where}\quad
\imf{\Psi}{\lambda}=\alpha \lambda+\lambda^2\sigma^2/2+\int_{[-1,1]} [ e^{\lambda x}-1-\lambda x ]\, \imf{\pi}{dx}
\]by the L\'evy-Kintchine formula. 
Let $X$ be a L\'evy process whose paths have infinite variation; equivalently, we assume that
\[
\sigma^2>0 \quad \text{or}\quad\int \abs{x}\, \imf{\pi}{dx}=\infty. 
\]

The L\'evy measure $\pi$, which is concentrated on $[-1,1]$, 
satisfies
\[ 
\int x^2\, \imf{\pi}{dx}<\infty. 
\]In other words, the characteristic triplet of $X$ is $(\alpha, \sigma, \pi)$. 


The following result is a trivial extension of the well-known Esscher change of measure for L\'evy processes, as found in \cite{MR3155252}. It will imply that the superior and inferior limits in Theorem \ref{RogozinForEITheorem} are not finite. Let $\F_t=\sag{X_s:s\leq t}$. 
\begin{proposition}[Esscher transform]
\label{EsscherCMProposition}
Fix $\theta\in\re$. 
Define the measure $\q$ by its restriction to $\F_t$ by
\[
\q|_{\F_t}=e^{\theta X_t-t\imf{\Psi}{\theta}}\cdot \p|_{\F_t}. 
\]Then, under $\q$, the stochastic process $X$ is a L\'evy process whose Laplace exponent is:
\[
\imf{\Psi^\theta}{\lambda}=\imf{\Psi}{\lambda+\theta}-\imf{\Psi}{\theta}. 
\]In particular, the characteristic triplet of $X$ under $\q$ is $(\alpha_\theta, \sigma, \pi_\theta)$ where:
\[
\alpha_\theta=\alpha+\theta\sigma^2+ \int_{[-1,1]} \bra{e^{\theta x}-1}x\, \imf{\pi}{dx}
\quad \text{and} \quad
\imf{\pi_\theta}{dx}=\imf{\indi{[-1,1]}}{x} e^{\theta x}\, \imf{\pi}{dx}. 
\]
\end{proposition}
Note that $\alpha_\theta\to -\infty$ as $\theta\to -\infty$ when $\int \abs{x}\, \imf{\pi}{dx}=\infty$ or $\sigma^2>0$. 


We now specialize to the spectrally positive case  and then use a (simple) argument to deduce the general case. 
\subsection{The spectrally positive case}
\label{subsection:sprogozinLP}
We now focus on the spectrally positive case, which corresponds to when $\pi$ is concentrated on $[0,1]$. 
When $X$ is spectrally positive, a general result of Fristedt implies that $\limsup_{t\to 0}\abs{X_t}/t=\infty$; cf. the proof of part A of Theorem 1 in \cite{MR0288841}. 
Since $X_t/t$ is a reverse martingale with no negative jumps (when $t$ decreases) which does not converge (because of the preceeding phrase),  Proposition 7.19 in \cite{MR1876169} tells us that 
$\imf{\overline D}{X}=\infty$. 


To prove that $
\imf{\underline D}{X}
=-\infty$, we use the following result of Sato for spectrally positive L\'evy processes. 
\begin{lemma}
\label{SatosResultForSPLP}
If $X$ is a spectrally positive L\'evy process with parameters $(\alpha, \sigma, \pi)$ with jumps bounded by $1$ and $\alpha=\esp{X_1}\leq  0$ then $\proba{X_t\leq 0}\geq 1/16$ for all $t\geq 0$. 
Also, for any deterministic sequence $t_n\to 0$,
\[\liminf_{n\to\infty}\frac{X_{t_n}}{t_n}\leq 0.\] 
\end{lemma}
The first statement is Proposition 46.8 in \cite{MR1739520}; 
the proof is simple and  based on inequalities of the Paley-Zygmund type 
(based on exponentials of $X^1$) and on properties of the Laplace exponent. 
However, it should be noted that the theorem cannot hold for all $\alpha$ (consider $\alpha\to\infty$, so that $\proba{X_t\leq 0}\to 0$), and that the proof is valid when $\alpha\leq 0$. 
The second statement is found in penultimate paragraph of the proof of Theorem 47.1 of \cite{MR1739520} and basically follows from the Borel-Cantelli lemma and the first part of Lemma \ref{SatosResultForSPLP}; however, in the proof, one uses that $\alpha\leq 0$, so that it remains valid. 
We reprove the lemma in the EI setting using Lemma \ref{lemma:simpleAndMoreGeneralSato} below. 

\begin{proof}[Proof of Theorem 1 for totally asymmetric L\'evy processes]
As before, we restrict ourselves to the spectrally positive case with jumps bounded by $1$. 

Note in particular, that the above result implies that $\imf{\underline D}{X}\leq 0$ for any spectrally positive L\'evy as in its statement. On the other hand, by absolute continuity, we see that $\imf{\underline D}{X}$ takes the same constant value both under $\p$ and under $\q$. Hence, we can write
\[
\underline D(X)=\alpha_\theta+ \underline D(\tilde X)
\]where $\tilde X$ is a spectrally positive L\'evy process of characteristics $(0,\sigma, \pi_\theta)$. 
By the preceeding lemma, $\underline D(\tilde X)\leq 0$. As we remarked, $\alpha_\theta\to -\infty$ as $\theta\to -\infty$, so that $\underline D(X)=-\infty$. 
We deduce that for all $\alpha\in\re$, $\underline D(X+\alpha\id)=-\infty$. 
\end{proof}


\subsection{The general case}
\label{subsection:RogozinForLPGeneralCase}
Let $X$ be a L\'evy process of infinite variation and bounded jumps. 
If suffices to prove $\imf{\underline D}{X}=-\infty$ for any such process and then apply this to  $-X$ to conclude that also $\imf{\overline D}{X}=\infty$. 
Using the L\'evy-It\^o decomposition, we write it as $X=X^\text{neg}+X^\text{pos}$ where $X^i$ are independent L\'evy processes, where $X^\text{neg}$ is spectrally negative and $X^\text{pos}$ is spectrally positive; this can be achieved with $\esp{X^\text{pos}_t}=0$ so that Lemma \ref{SatosResultForSPLP} applies. 
In particular, from Theorem \ref{RogozinForEITheorem} for totally asymmetric L\'evy processes (proved in Subsection \ref{subsection:sprogozinLP}),$\imf{\underline D}{X^\text{neg}}=-\infty$. 
Hence, there exists a random sequence $V_n\downarrow 0$ such that $X^\text{neg}_{V_n}\leq -n V_n$. 
Since $X^\text{neg}$ is independent of $X^\text{pos}$ and $X^\text{pos}$ is spectrally positive, Lemma \ref{SatosResultForSPLP} implies that $\liminf_{n\to\infty} X^\text{pos}_{V_n}/V_n\leq 0$. We then conclude that
\[
\liminf_{t\to 0} \frac{X_t}{t}
\leq \liminf_{n\to\infty} \frac{X_{V_n}}{V_n}
\leq  \liminf_{n\to \infty} \frac{X^\text{pos}_{V_n}}{V_n}-n
=-\infty. 
\]


\section{Dini derivatives of EI processes in the totally asymmetric case}
\label{totallyAsymmetricDiniSection}

In this section, we prove Theorem \ref{RogozinForEITheorem}; 
it suffices to prove it for extremal EI process and obtain the general case by mixing. 
For concreteness, we assume that $X$ is extremal and only has positive jumps, so that $\beta_i\geq 0$. 

We first show that Dini derivatives are constant. 
\begin{proposition}\label{prop:ZeroOneLawForDiniDerivative}
Let $X$ be an extremal  EI process of parameters $(\alpha, 0,\beta)$. 
Then
\[
\imf{\underline D}{X}=c_1\quad\text{and}\quad\imf{\overline D}{X}=c_2
\]%
for some constants $c_1,c_2\in [-\infty,\infty ]$.
\end{proposition}
\begin{proof}
Fix any $k\in \na$, and define $X^{(-k)}_t=\alpha t+\sigma b_t+\sum_{j=k+1}^\infty \beta_j\bra{\indi{U_j\leq t}-t}$. 
Hence, for any $t<\min\{U_i:i\in [k] \}$ we have
\begin{lesn}
	\frac{X_t}{t}=\frac{X^{(-k)}_t}{t}-\sum_1^k\beta_i,
\end{lesn}which implies
\begin{equation}
\label{equation:DiniDerivativesForZeroOneProposition}
	\varlimsup_{t\downarrow 0}\frac{X_t}{t}=\varlimsup_{t\downarrow 0}\frac{X^{(-k)}_t}{t}-\sum_1^k\beta_i,
\end{equation}for any $k$. 
Let
\[
\G=\bigcap_{\eps\in (0,1)}\sag{ b_s:s\leq \eps}. 
\]The (local) absolute continuity of the Brownian bridge with respect to Brownian motion, and the Blumenthal zero-one law for the latter imply that $\G$ is trivial. 
Let $\F_k= \sag{U_k,U_{k+1},\ldots}$ and note that $\G$ is independent of (any \sa\ but in particular) $\F_k$. 
As noted in the proof of \cite[Prop. I\S2.4]{MR1406564}, the argument for Kolmogorov's zero one law tells us that $\bigcap_k \G\vee\F_k$ is trivial. 
Since the right-hand side of \eqref{equation:DiniDerivativesForZeroOneProposition} is $\G\vee \F_k$-measurable, we deduce that the left-hand side is $\bigcap_k \G\vee\F_k$-measurable and therefore trivial. 
A similar argument works for the lower Dini derivative. 
\end{proof}

We will proceed as in the case of L\'evy processes: 
we first give a change of measure for EI processes, 
analogous to the Esscher transformation, 
which has the effect of transforming the drift and the jumps. 
As in the L\'evy case, $\overline{D}(X)=\infty$ follows from simple results of the literature. 
We then use martingale arguments to prove that $\underline{D}(X)\leq 0$. 
Finally, our change of measure will imply that $\underline{D}(X)=-\infty$. 

\begin{proposition}[Change of measure]
\label{changeOfMeasureProposition}
Let $X=(X_t,t\in [0,1])$ be an extremal EI process with (deterministic) characteristics $(\alpha, \sigma,\beta)$, defined on the probability space $\ofp$. 
Then
\[
\esp{e^{\theta X_t}}=e^{\theta^2\sigma^2 t(1-t)/2}\prod_{j=1}^\infty \bra{e^{-\theta\beta_j t}\bra{1-t}+e^{\theta\beta_j (1-t)}t}<\infty
\]for all $t\in [0,1]$ and $\theta\in\re$.  

Fix $T\in (0,1)$, $\theta\in\re$ and let $\F_t=\sag{X_s:s\leq t}$. 
Define $\q$ on $\F_T$ by\begin{esn}
\imf{\q}{A}=\frac{\imf{\se_\p}{\indi{A}e^{\theta X_T}}}{\esp{e^{\theta X_t}}}.
\end{esn}Under $\q$, the stochastic process $(X_t,t\leq T)$ is an EI process whose (random) characteristics $(\alpha^\theta,\sigma,\beta^\theta)$ have the following law. 
Let $(B_j)$ be independent Bernoulli random random variables with parameter $p_j$ given by
\begin{equation}
\label{p_jDef}
p_j=\frac{Te^{\theta b_j}}{Te^{\theta b_j}+(1-T)}. 
\end{equation}Then 
\[
\alpha^\theta=\alpha T+\theta\sigma+\sum_{j}\beta_j \bra{B_j-T}
\quad\text{and}\quad
\beta^\theta_j=\beta_j B_j,
\]where $\sum_{j}\beta_j \bra{B_j-T}$ converges almost surely and in $L_1$. 

If $\alpha=0$ and either $\sigma>0$ or $\sum_i\indi{\beta_i\neq 0}=\infty$ 
then $\esp{e^{\lambda X_t}}\to\infty$ as $\lambda\to\infty$. 
\end{proposition}

\begin{remark}
When $X$ is a finite variation EI process, driven by the two parameters $(\tilde \alpha,\beta)$ (rather than $(\alpha,0,\beta)$, ) as explained in Section \ref{resultsSection,whichIsIntroduction}, then $X$ under $\q$ is also of finite variation and is driven by the two parameters $(\tilde\alpha T,\beta^\theta)$. 
Hence, $\imf{\overline D}{X^\theta}=\imf{\underline D}{X^\theta}=\tilde\alpha$, which does not depend on $\theta$; the interpretation is that, in this case, the change of measure is not adding a drift. 
If we choose not to reparametrize with $(\tilde\alpha,\beta^\theta)$, the finite variation case is characterized by the fact that $\alpha^\theta$ is bounded in $\theta$. On the other hand, if $X$ is of infinite variation, we shall see that $\alpha^\theta$ stochastically increases from $-\infty$ to $\infty$. 
\end{remark}

\begin{proof}
Let us begin with the finitude of the moment generating function of $X$. 
Use the canonical representation $X_t=\alpha t+\sigma b_t+\sum_i\beta_j\bra{\indi{U_j\leq t}-t}$. 
Define
\[
\imf{\phi}{t,\theta,\beta}=e^{-\theta\beta_j t}\bra{1-t}+e^{\theta\beta_j (1-t)}t.
\]
Note that, 
\[ 
\esp{e^{\theta\bra{\beta_j\indi{U_j\leq t}-t}}}
=\imf{\phi}{t,\theta,\beta_j}. 
\]For fixed $t,\theta$, we have\[
\imf{\phi}{t,\theta,\beta}=1+\frac{1}{2}  \left(\theta ^2 t-\theta ^2 t^2\right)\beta^2+O\left(\beta^3\right)
\]as $\beta\to 0$. 
Therefore, $\prod_{j=1}^\infty   \imf{\phi}{t,\theta,\beta_j}$ exists, since $\beta$ is square summable. 
By Fatou's lemma, we see that
\[
\esp{e^{\theta X_t}}
\leq e^{\alpha \theta t+\theta^2\sigma^2 t(1-t)} \lim_{n\to\infty} 
\prod_{j\leq n} \esp{e^{\theta\bra{\beta_j\indi{U_j\leq t}-t}}}
\leq e^{\alpha \theta t+\theta^2\sigma^2 t(1-t)} 
\prod_{j=1}^\infty   \imf{\phi}{t,\theta,\beta_j}<\infty. 
\]But then, H\"older's inequality implies the log-concavity of $\theta\mapsto \esp{e^{\theta X_t}}$, which is then enough to obtain uniform integrability of the sequence $X^n=X-\sum_{j>n} \beta_j\bra{\indi{U_j\leq t}-t}$, which then implies the stated infinite product formula for the moment generating function of $X_t$. 

Consider now a sequence $V=(V_j)$ of independent uniform $(0,1)$ random variables, which is independent of $b$ and the $(U_j)$  and define $B_j=\indi{V_j\leq p_j}$. 
Note that obviously $\sum \beta_j^2 B_j^2<\infty$. 
Regarding $\alpha^\theta$, we use the Kolmogorov three series theorem. 
Indeed, since the sequence $(\beta_j)$ is bounded, so is the sequence $(\beta_j \bra{B_j-T})$. 
On the other hand, we have
\[
\esp{\sum_{j=1}^n\beta_j \bra{B_j-T}}
=\sum_{j=1}^n T(1-T)\beta_j\frac{e^{\theta\beta_j}-1}{Te^{\theta \beta_j}+(1-T)}
=\imf{O}{\sum_{j=1}^\infty \beta_j^2}.
\]Finally, we see that
\[
\var{\sum_{j=1}^n\beta_j \bra{B_j-T}}=\sum_{j=1}^n\beta_j^2 p_j(1-p_j)\leq \sum_{j=1}^\infty \beta_j^2. 
\]

Consider also the process
\[
X^\theta_t=\alpha^\theta t/T+ \theta\sigma b_{t/T}+ \sum_j \beta_jB_j\bra{\indi{U_j\leq t/T}-t/T}
\]defined on $[0,T]$. Note that $X^\theta$ is an EI process on $[0,T]$ with random characteristics $(\alpha^\theta,\sigma^\theta,\beta^\theta)$; we finish the proof by comparing, through moment generating functions, 
the finite-dimensional distributions of the increments of $X$ under $\q$ and of $X^\theta$ (under $\p$). 

First of all, by independence of $U$ and $b$; since the law of $b$ under $e^{\theta b_T}\cdot \p$ equals that of $b+\sigma\theta \id$ (as can be proved through the Gaussian character of $b$), it suffices to prove the theorem when $\sigma=0$. Since $\alpha$ is deterministic, it also suffices to consider $\alpha=0$. 

Let $0=t_0\leq t_1\leq\cdots\leq t_n=T$ and $\lambda_1\cdots,\lambda_n\in\re$. 
Using similar arguments as for justifying the exchange of expectation and infinite products in the computation of the generating function, we first see that
\begin{align*}
&\imf{\se}{\prod_{i=1}^n e^{\lambda_i\bra{ X_{t_i}-X_{t_{i-1}}}}e^{\theta X_T}}
\\&=\prod_j \imf{\se}{ \prod_{i=1}^ne^{\lambda_i\bra{ \alpha (t_i-t_{i-1})+\sigma (b_{t_i}-b_{t_{i-1}})+\beta_j\indi{t_{i-1}\leq U_j\leq t_i} }}e^{ \theta \bra{ \alpha T+\sigma b_T+\beta_j\indi{ U_j\leq T} } }}.
\end{align*}Therefore, the Laplace transform of $(X_{t_1},\ldots, X_{t_n})$ is finite under $\q$ and
\begin{align*}
\imf{\se_\q}{\prod_{i=1}^n e^{\lambda_i\bra{ X_{t_i}-X_{t_{i-1}}}}}
&=\frac{1}{\imf{\se_\p}{e^{\theta X_T}}}\imf{\se_\p}{\prod_{i=1}^n e^{\lambda_i\bra{ X_{t_i}-X_{t_{i-1}}}}e^{\theta X_T}}
\\&=\frac{1}{\imf{\se_\p}{e^{\theta X_T}}}\prod_{j=1}^\infty \imf{\se_\p}{\prod_{i=1}^n e^{\lambda_i\beta_j[  \indi{U_j\in [t_{i-1},t_i]}-(t_i-t_{i-1}) ]}e^{\theta \beta_j[ \indi{U_j\in [0,T]}-T]}}. 
\end{align*}By considering the interval of the partition $[t_{i-1},t_i]$ on which $U_j$ falls, and recalling the definition of $p_j$ in \eqref{p_jDef}, we get\begin{align*}
&\imf{\se_\q}{\prod_{i=1}^n e^{\lambda_i\bra{ X_{t_i}-X_{t_{i-1}}}}}
\\&=\frac{1}{\imf{\se_\p}{e^{\theta X_T}}}\prod_{j=1}^\infty e^{-\theta\beta_jT-\beta_j \sum_{i=1}^n\lambda_i(t_i-t_{i-1})}\bra{(1-T)+\sum_{i=1}^n (t_i-t_{i-1}) e^{(\lambda_i+\theta)\beta_j} }
\\&=\prod_{j=1}^\infty e^{-\beta_j \sum_{i=1}^n\lambda_i(t_i-t_{i-1})}\bra{(1-p_j)+\sum_{i=1}^n\frac{t_i-t_{i-1}}{T}p_j  e^{\lambda_i\beta_j} }. 
\end{align*}

Recall that $\alpha$ and $\sigma$ are zero in the definition of $\alpha^\theta$ and $X^\theta$. 
On the other hand, using the definition of $X^\theta$, 
we can use the distributional assumptions on $B$ and $U$ (first independence, then conditioning on $B$, and finally considering the inverval $[t_{i-1},t_i)$ on which $U_j$ falls) to obtain
\begin{align*}
&\prod_j \esp{ \prod_{i=1}^n 
e^{\lambda_i \beta_j \bra{B_j-T}\frac{t_i-t_{i-1}}{T} +\lambda_i\beta_jB_j\bra{\indi{TU_j\in (t_{i-1},t_{i}) }-\frac{t_i-t_{i-1}}{T}} } 
}
\\&= \prod_j e^{-\sum_{i=1}^n\lambda_i\beta_j(t_i-t_{i-1})}  \esp{  \bra{ \paren{1-p_j}+p_j e^{\sum_{i=1}^n\lambda_i\beta_j\indi{TU_j\in (t_{i-1},t_i)}} }}
\\&=\prod_j  e^{-\sum_{i=1}^n\lambda_i\beta_j(t_i-t_{i-1})}  \bra{(1-p_j)+p_j\sum_{i=1}^n e^{\lambda_i\beta_j } \frac{t_i-t_{i-1}}{T} }
\end{align*}The preceding equation shows that the increments of the left-hand side 
have the same law as the corresponding increments of $\beta_j[\indi{U_j\leq \cdot}-\cdot]$ under $\q$, for every $j$. 
Thus, 
\[
\prod_j \esp{ \prod_{i=1}^n 
e^{\lambda_i \beta_j \bra{B_j-T}\frac{t_i-t_{i-1}}{T} +\lambda_i\beta_jB_j\bra{\indi{TU_j\in (t_{i-1},t_{i}) }-\frac{t_i-t_{i-1}}{T}} } 
}=\imf{\se_\q}{\prod_{i=1}^n e^{\lambda_i\bra{ X_{t_i}-X_{t_{i-1}}}}}<\infty.
\]
Similarly as in the proof of Proposition \ref{changeOfMeasureProposition}, using Fatou's lemma and H\"older's inequality we deduce
\[
\esp{\prod_{i=1}^ne^{\lambda_i [X^\theta_{t_i}-X^\theta_{t_{i-1}}]}}=\prod_j \esp{ \prod_{i=1}^n 
e^{\lambda_i \beta_j \bra{B_j-T}\frac{t_i-t_{i-1}}{T} +\lambda_i\beta_jB_j\bra{\indi{TU_j\in (t_{i-1},t_{i}) }-\frac{t_i-t_{i-1}}{T}} } 
}. 
\]Hence the finite-dimensional distributions of $X$ under $\q$ and $X^\theta$ under $\p$ are the same. 

The last part of the statement follows from the fact that if $\xi$ is any random variable on $\re$ with finite generating function $g$, then $g(\infty)=\infty$ whenever $\proba{\xi>0}>0$. 
The hypotheses on $X$ are chosen so that $\proba{X_t>0}>0$ for all $t\in (0,1)$. 
Indeed, when $\alpha=0$, $(X_t/(1-t),t<1)$ is a martingale; the assumption $\proba{X_t>0}=0$ implies $\esp{X_t^-}=0$ which then gives $\proba{X_t=0}=1$ and our hypotheses imply that $X_t$ has no atoms for $t\in (0,1)$ as shown in the proof of Lemma 1.2 in \cite{MR1417982}. 
\end{proof}

We now consider the behavior of the drift $\alpha^\theta$ as a function of $\theta$. 
\begin{proposition}
\label{propositionPropertiesOfAlphaTheta}
%
The mapping $\theta\mapsto \alpha^\theta$ is stochastically non-decreasing.  
Also, $\alpha^\theta\in L_1$, $\theta\mapsto\esp{\alpha^\theta}$ is continuous and strictly increasing and, if $X$ is of infinite variation,  $\esp{\alpha^\theta}\to\pm\infty$ as $\theta\to\pm\infty$. 
\end{proposition}
\begin{proof}
We have already proved that $\alpha^\theta$ is a convergent series (plus the couple of constants $\alpha$ and $\theta\sigma$); it is absolutely divergent in the infinite variation case and otherwise absolutely convergent. 
Using our explicit construction of the random variables $B_j$ as $\indi{V_j\leq p_j}$ and the definiton of $p_j$ in \eqref{p_jDef},  we note that the $\beta_j [B_j-T]$ are increasing in $\theta$, which implies the same for $\alpha^\theta$. 

Recall that $\alpha^\theta$  is (modulo a constant) a series of independent random variables taking two values, 
whose means and variances are summable. Hence $\alpha^\theta\in L_1$ and 
\[
\esp{\alpha^\theta}=\alpha+\theta\sigma+ T(1-T)\sum_j \beta_j \frac{e^{\theta \beta_j}-1}{Te^{\theta \beta_j}+(1-T)}. 
\]The above summands are $\imf{O}{\theta \beta_j^2}$, uniformly for $\theta$ on compact sets. 
This implies the continuity of $\theta\mapsto\esp{\alpha^\theta}$. 
But the mapping\[
\theta\mapsto \beta_j \frac{{e^{\theta \beta_j}-1}}{{Te^{\theta \beta_j}+(1-T)}}
\]is strictly increasing, and monotone convergence inplies the same for $\theta\mapsto\esp{\alpha^\theta}$. 
Finally, note that the preceding function of $\theta$ goes to $ \pm\beta_j$ as $\theta\to\pm\infty$. 
When $X$ is of infinite variation,  Fatou's lemma can be applied to the series for $\esp{\alpha^\theta}$, as the summands in its definition it are either all positive or all negative, and conclude that $\esp{\alpha^\theta}\to \pm\infty$ as $\theta\to\pm\infty$. 
\end{proof}


We now give a version of Lemma \ref{SatosResultForSPLP} for EI processes, as well as a simple lemma which uses it.

\begin{lemma}
\label{lemma:lowerBoundProbabilityEI}
Let $X$ be an extremal EI process with parameters $(0,\sigma,\beta)$ such that $\beta_i\leq 0$ for all $i$. 
Then $\proba{X_t\geq 0}\geq 1/16$ for every $t\leq 1/2$. 
\end{lemma}

\begin{proof}
Assume we have proved that
\begin{equation}\label{eqnPsi2LambdaAnd4PsiLambda}
\esp{e^{2\lambda X_t}}\leq \esp{e^{\lambda X_t}}^4
\end{equation}for $\lambda>0$ and $t\leq 1/2$. 
Then, using  the Cauchy-Schwarz inequality we would obtain
\begin{align*}
\proba{X_t\geq 0}&\geq \frac{\paren{\esp{e^{\lambda X_t}}-\esp{e^{\lambda X_t};X_t\leq 0}}^2}{\esp{e^{2\lambda X_t}}}\\
& \geq \frac{\paren{\esp{e^{\lambda X_t}}-1}^2}{\esp{e^{\lambda X_t}}^4}
\end{align*}for $\lambda>0$. 

By Proposition \ref{changeOfMeasureProposition} 
we can chose $\lambda_t$ such that $\esp{e^{\lambda_t X_t}}=2$, which implies
\begin{lesn}
	\proba{X_t\geq 0}\geq \frac{1}{16}.
\end{lesn}Now, let us prove \eqref{eqnPsi2LambdaAnd4PsiLambda}. 
First note that \eqref{eqnPsi2LambdaAnd4PsiLambda} is an equality for a Brownian bridge.
Hence, by the independence of the latter with the (purely discontinuous) jump part of $X$, 
it is enough to assume $\sigma=0$.
Defining
\begin{lesn}
	\phi_j(\lambda):=\ln \esp{e^{\lambda \beta_j[\indi{U_j\leq t}-t]}}:=\ln \psi_j(\lambda),
\end{lesn}it is enough to prove for every $j\in\na$ that
\begin{equation}\label{eqnInequalityForLaplaceExponentsOfEIs}
	0\leq 4\phi_j(\lambda)-\phi_j(2\lambda)\ \ \ \ \ \ \ \ \ \lambda\geq 0.
\end{equation}Since the right-hand side is zero when $\lambda=0$, proving it has a non-negative derivative implies Equation \eqref{eqnInequalityForLaplaceExponentsOfEIs}.
Taking the derivative with respect to $\lambda$, we need to prove that
\begin{align*}
0\leq &4\frac{\beta_j(1-t)te^{\lambda \beta_j(1-t)}+\beta_j(-t)(1-t)e^{\lambda \beta_j(-t)}}{\psi_j(\lambda)}
\\&-2\frac{\beta_j(1-t)te^{2\lambda \beta_j(1-t)}+\beta_j(-t)(1-t)e^{2\lambda \beta_j(-t)}}{\psi_j(2\lambda)}
\end{align*}which is equivalent to 
\[
0\geq 2\frac{e^{\lambda \beta_j}-1}{te^{\lambda \beta_j}+1-t}-\frac{e^{2\lambda \beta_j}-1}{te^{2\lambda \beta_j}+1-t}.
\]and further equivalent, since the denominators are positive by convexity of the exponential function, to 
\[
1-t+(t-2)e^{\lambda \beta_j}+(1+t)e^{2\lambda \beta_j}-te^{3\lambda \beta_j}\geq 0
\]As before, the left-hand side at $\lambda=0$ is zero, hence, it suffices to prove its derivative is non-negative.
We apply an analogous  reasoning by evaluation at $\lambda=0$, differentiation and division by $\beta_je^{\lambda \beta_j}$ (which is negative) three times!
The sequence of derivatives, taking out the factor $\beta_je^{\lambda \beta_j}$ are
\begin{align*}
& t-2+2(1+t)e^{\lambda \beta_j}-3te^{2\lambda \beta_j}
,\\
& 2(1+t)-6te^{\lambda \beta_j}
\text{ and}\\
& 
-6te^{\lambda \beta_j}\leq 0
.
\end{align*}The penultimate function is non-negative at $\lambda=0$ when $t\in [0,1/2]$. 
The last inequality then shows that the penultimate one is non-negative, which we can then bootstrap to show inequality \eqref{eqnInequalityForLaplaceExponentsOfEIs}. 
\end{proof}

The choose of $\lambda_t$ such that $f(\lambda_t)=2$ in the preceding proof, seems arbitrary; 
the reader can check it gives the best bound obtainable by this method.

\begin{lemma}
\label{lemma:simpleAndMoreGeneralSato}
Let $X=(X_t,t\geq 0)$ be a \cadlag\ process such that, 
for some sequence $t_n\downarrow 0$, 
the random variable $\liminf_n X_{t_n}/t_n$ is constant. 
Assume that, for some $\eps, c>0$,  $\proba{X_t\leq  0}>c$ for every $t\in [0,\eps]$. 
Then, $\liminf_n X_{t_n}/t_n\leq 0$. 
\end{lemma}

\begin{remark}
Note that the above can be applied when the augmented initial \sa\ of $X$, given by $\cap_{\eps>0}\sag{X_s:s\leq \eps}$, is trivial (hence for Feller processes) or in the case of extremal EI processes by mimicking the proof of Proposition \ref{prop:ZeroOneLawForDiniDerivative}; in this case, we can apply the result to any sequence $t_n\downarrow 0$. 

Also, by mixing,we deduce from the above two lemmas that if $X$ is an EI process with random characteristics $(\alpha,\sigma,\beta)$, 
where $\alpha\leq 0$ and $\beta_i\geq 0$ almost surely, then $\liminf_{n} X_{t_n}/t_n\leq 0$ almost surely for any sequence $t_n\downarrow 0$. 
\end{remark}

\begin{proof}
By contrapositive, assume that $\liminf_n X_{t_n}/t_n>0$ almost surely. 
Then:
\begin{align*}
\lim_{N\to \infty }\proba{X_{t_N}\leq 0}&\leq \lim_{N\to \infty }\proba{X_{t_n}\leq 0 \text{ for some }n\geq N}
\\ &= \proba{X_{t_n}\leq 0 \text{ infinitely often}}
\\&=\proba{\liminf_{n} \frac{X_{t_n}}{t_n}\leq 0}=0. \qedhere
\end{align*}
\end{proof}

We are now ready to prove Theorem \ref{RogozinForEITheorem} in the case of totally asymmetric EI processes. 
The proof is similar to the one for (spectrally positive) L\'evy processes given in Section \ref{LevyProcessSection}. 

\begin{proof}[Proof of Theorem \ref{RogozinForEITheorem} when $\alpha,\sigma=0$ and $\beta_i\geq 0$]
We first prove that $\imf{\overline D}{X}=\infty$. 
Define the measure $\beta(dx)=\sum \delta_{\beta_j\in dx}$.
Using Fubini's theorem
\begin{lesn}
	\int_0^{1}\beta(t,\infty)dt=\int_0^{1}\int_t^\infty\beta(dx)\,dt=\int_0^\infty\int_0^{1\wedge x}\,dt\,\beta(dx)=\sum \beta_j\wedge 1=\infty.
\end{lesn}Hence, we obtain for any $k\geq 1$
\begin{lesn}
	\int_0^{1}\beta(kt,\infty)dt=\int_0^{k}\beta(t,\infty)du/k=\infty.
\end{lesn}It follows that, for infinitely many $t\in (0,1)$ we have $|\Delta X_t|>kt$, which in turn implies $|X_t|>kt/2$ or $|X_{t-}|>kt/2$, for such $t$.
Since $k$ was arbitrary, we have $\varlimsup_{t\downarrow 0} |X_t|/t=\infty$.

Note that $(X_t/t,t\in (0,1])$ is a backward martingale (here, it is important that we restricted ourselves to the case $\alpha=0$) without positive jumps. Note that it does not converge, thanks to the preceding paragraph. 
The process $N=(-X_{-t}/t,t\in [-1,0))$ is therefore a martingale without positive jumps (divergent almost surely). 
If we define $\tau_c$ as the first time $N$ reaches $c\in \re_+$, then $(c-N_{-t\wedge \tau_c},-1\leq t<0)$ is a non-negative martingale. 
By the martingale convergence theorem $N_{\cdot\wedge \tau_c}$ converges a.s. to a finite limit as $t\uparrow 0$.
If $\tau_c$ were infinite, $N$ itself would therefore converge; 
hence, $\tau_c$ is almost surely finite. 
Since $c$ was arbitrary, then $\varlimsup_{t\uparrow 0} N_{-t}=\infty$, which implies $\varlimsup_{t\downarrow 0} X_t/t=\infty$. (The above argument was taken from \cite{MR1825153} and \cite{MR0288841}). 
We have proved that $\imf{\overline D}{X}=\infty$ for any extremal EI process with parameters $(\alpha, 0,\beta)$  of infinite variation when $\beta_i\geq 0$ for all $i$. 
Taking mixtures, we can let $\alpha$ and $\beta$ be random, as long as $\sum_i \beta_i=\infty$ almost surely. 
We use this remark in the following paragraph. 

We now prove that $\underline{D}{X}=-\infty$. 
We now apply a change of measure (through Proposition \ref{changeOfMeasureProposition}) to $X$; call the resulting measure $\q^\theta$ to stress the dependence on $\theta$. 
Write $\alpha^\theta \id/T+Y^\theta$ for the (random parameter) EI process whose law is $\q^\theta$. 
Recall from Proposition \ref{prop:ZeroOneLawForDiniDerivative} that $\underline{D}{X}$ is a constant. 
Since $\q^\theta$ is absolutely continuous with respect to $\p$ then $\imf{\underline{D}}{X}=\alpha^\theta/T+\imf{\underline{D}}{Y^{\theta}}$. 
Even if $Y^{\theta}$ has random parameters, they jumps are almost surely positive. 
The remark following Lemmas \ref{lemma:lowerBoundProbabilityEI} and \ref{lemma:simpleAndMoreGeneralSato} implies that $\imf{\underline{D}}{Y^{\theta}}\leq 0$ almost surely. 
Taking expectations we see that
\[
\imf{\underline{D}}{X}=\esp{\imf{\underline{D}}{X}}=\esp{\alpha^\theta/T}+\esp{\imf{\underline{D}}{Y^{\theta}}}\leq \esp{\alpha^\theta/T}\to -\infty
\]as $\theta\to-\infty$, since $\esp{\alpha^\theta/T}\to-\infty$ by Proposition \ref{propositionPropertiesOfAlphaTheta}. 
\end{proof}

\begin{proof}[Proof of Theorem \ref{RogozinForEITheorem}]
As before, assume that $\alpha=0$ and focus only on the statement $\underline D(X)=\infty$, since we can at the end apply it to $-X$. 
Write $X=X^p+X^n$, where $X^p$ and $X^n$ are independent extremal EI processes with parameters $(0,\sigma,\beta^p)$  and $(0,0,\beta^n)$, where $\beta^p$ are the positive terms of $\beta$ and $\beta^-$ the negative ones. 
We have proved Theorem \ref{RogozinForEITheorem} for $X^p$ and for $X^n$ if they are of infinite variation. 
If one of them is of finite variation, then the other one must be of infinite variation, and then Theorem \ref{RogozinForEITheorem} holds for $X$. 
Hence, we can assume that both $X^p$ and $X^n$ are of infinite variation. 

But then, there exists a random sequence $T_n\downarrow 0$ such that $X^n_{T_n}/T_n\to -\infty$ 
thanks to Theorem \ref{RogozinForEITheorem} for spectrally negative EI processes (just proved). 
Since $(T_n)$ is independent of $X^p$, we can apply 
Lemmas \ref{lemma:lowerBoundProbabilityEI} and $\ref{lemma:simpleAndMoreGeneralSato}$ to conclude that $\liminf_{n} X^p_{T_n}/T_n\leq 0$. 
We conclude that $\underline D(X)\leq \liminf_n X_{T_n}/T_n =-\infty$. 
\end{proof}


\section{Further applications}
\label{furtherApplicationSection}

We now move on to the applications of Theorem \ref{RogozinForEITheorem}, which were stated as Theorems \ref{ZeroOneLawTheorem}, \ref{DIMTheorem} and \ref{ConMinTheorem}. 
We already mentioned that Theorem \ref{DIMTheorem} follows from the same arguments as in \cite{MR3558138} once we have Theorem \ref{RogozinForEITheorem}. 
Again, it suffices to prove the theorems for extremal EI processes. 

\subsection{An extension of Millar's zero-one law at the minimum }
To prove Theorem \ref{ZeroOneLawTheorem}, 
we will use a representation of the post minimum process associated to an EI process found in \cite{MR1232850} which is now recalled.

Let $X$ be an extremal EI process with parameters $(\alpha,\sigma,\beta)$; according to \cite[Ch. 2]{MR2161313}, such a process is a semimartingale. 
Let $\tau=\sup\{t\in [0,1]:\underline{X}_1=X_t\wedge X_{t-} \}$ be the time of the ultimate infimum.
Define the post-infimum process $\underrightarrow{X}$ as
\[ \underrightarrow{X}(t)=\begin{cases} 
X_{\tau+t}-\underline{X}_1 & t\leq 1-\tau \\
\dagger & t>1-\tau,
\end{cases}
\](where $\dagger$ is a cemetery state) and the reversed pre-infimum process $\underleftarrow{X}$ as
\[ \underleftarrow{X}(t)=\begin{cases} 
X_{(\tau-t)-}-\underline{X}_1 & t\leq  \tau \\
\dagger & t>\tau.
\end{cases}
\]We introduce two processes $X^\uparrow$ and $X^\downarrow$ as in \cite{MR1232850}. 
Since $X$ is a semimartingale, it has a semimartingale local time at zero denoted $L$; 
this local time is actually zero unless $\sigma>0$ in which case
\[
L_t=\lim_{\eps\downarrow 0}\frac{\sigma^2}{2\eps}\int_0^t \indi{\abs{X_s}\leq \eps}\, ds. 
\]
Consider the time spent at $(0,\infty)$ and $(-\infty,0]$ up to time $t$ of $X$, that is
\begin{lesn}
	A^+_t=\int_0^t\indi{X_s>0}ds\ \ \ \ \mbox{ and }\ \ \ \ 	A^-_t=\int_0^t\indi{X_s\leq 0}ds,
\end{lesn}and consider also their right-continuous inverses $\alpha^\pm (t)=\inf\{ s:A^\pm_s>t\}$. 
It can be seen by a picture, that using the time change $\alpha^+$ on $X$ consist on erasing the jumps of $X$ that fall on $(-\infty,0]$ and closing up the gaps (similarly for $\alpha^-$). 
The process of juxtaposition of the excursions in $(0,\infty)$ is given by
\begin{lesn}
X^\uparrow(t)=\paren{X_\cdot+\sum_{0<s\leq \cdot}[\indi{X_s\leq 0}X^+_{s-}+\indi{X_s>0}X^-_{s-}]+L}(\alpha^+(t)). 
\end{lesn}We remark that an excursion in $(0,\infty)$ includes the possible initial positive jump across 0 and excludes the possible ultimate negative jump across 0.
The process of juxtaposition of the excursions in $(-\infty,0]$ is given by
\begin{lesn}
	X^\downarrow(t)=\paren{X_\cdot-\sum_{0<s\leq \cdot}[\indi{X_s\leq 0}X^+_{s-}+\indi{X_s>0}X^-_{s-}]-L}(\alpha^-(t)). 
\end{lesn}By establishing a bijection for discrete-time EI processes and passing to the limit, Bertoin obtains the following result. 
\begin{theorem}[Theorem 3.1 in \cite{MR1232850}]\label{teoBertoinsTheoremPostandPreInfimumProcesses}
Let $X$ be an extremal EI process on $[0,1]$ with parameters $(\alpha,\sigma,\beta)$. 
Then, the following equality in law holds
\begin{lesn}
(\underrightarrow{X},-\underleftarrow{X})\stackrel{d}{=}(X^\uparrow,X^\downarrow).
\end{lesn}
\end{theorem}

We first establish the following simple result for EI processes. 
\begin{lemma}\label{lemmaXIsIrregularAfterAJump}
	When $X$ satisfies UM, we have $\proba{X_{U_j}\text{ or }X_{U_j-}=0}=0$ for every $j\in \na$.
\end{lemma}
\begin{proof}
%
%
From the proof of Lemma 1.2 in \cite{MR1417982} we see that \defin{UM} implies that $\proba{X_t=x}=0$ for any $t\in (0,1)$ and $x\in \re$. 
Fix any $j\in \na$ and define $X^{j}_t=X_t-\beta_j[\indi{U_j\leq t}-t]$. 
Since $X_{U_j}=\beta_j(1-U_j)+X^{j}_{U_j}$, then
\begin{lesn}
	\proba{X_{U_j}=0}=\proba{\beta_j(1-U_j)+X^{j}_{U_j}=0}=\int_0^1\proba{\beta_j(1-t)+X^{j}_{t}=0}dt=0. 
\end{lesn}%
The statement for the left limit follows by time-reversibility. 
\end{proof}

\begin{proof}[Proof of Theorem \ref{ZeroOneLawTheorem}]
Let $X$ be an EI process satisfying \defin{UM}. 
The previous lemma tells us that $X$ does not jump into or from $0$. 

Assume that $X$ is irregular upward. 
Then, $X$ remains  negative up to the time
\[\tau_0^+=\inf\{t>0:X_t>0 \}, 
\]which is strictly positive. We actually have $\tau_0^+<1$ since otherwise the only positive value that $X$ has to take, by assumption,  would be taken at time $1$, by a jump,  and $X$ does not jump at $1$. 
The trajectory of $X$ up to $\tau_0^+$ might comprise several excursions below zero and we will be interested in the first one, which ends at the random time
\[
T=\min\set{t>0: X_t\geq 0}\leq \tau_0^+. 
\]Recall the definition of $\tau$ as the time of the last minimum. Let us prove that
\begin{equation}\label{eqnXJumpsFromItsInfimum}
	0<\Delta X_\tau\text{ almost surely.}
\end{equation}Note that
\begin{lesn}
	\Delta X_\tau=\Delta \underrightarrow{X}_0\stackrel{d}{=}\Delta X^\uparrow_0=\Delta X_{\tau_0^+},
\end{lesn}where the equality in distribution holds by Theorem \ref{teoBertoinsTheoremPostandPreInfimumProcesses}. 
Since we do not jump into or from $0$, then $\Delta X_{\tau_0^+}=0$ if and only if $X_{\tau_0^+}=0$ and $X$ is continuous at $\tau_0^+$. Assume that $\Delta X_{\tau_0^+}=0$ has positive probability. 
Then the reversed pre-infimum process would hit zero twice (and  the process $X$ would hit its infimum twice); this is impossible under \defin{UM}. 
Indeed, note that $X=X^\downarrow$ on $[0,T]$, and that from the construction of the pre-minimum process in Theorem \ref{teoBertoinsTheoremPostandPreInfimumProcesses}, $T$ has the same distribution as $S$ where
\[
S=\inf\set{t>0: \underleftarrow X_t=0}=\tau-\sup\set{s<\tau:  X_{t-}=\underline X_1}.
\]When $\Delta X_{\tau_0^+}=0$, then $T>0$ and $X_T=0$. 
Hence, with positive probability, we would have that $S<\tau$ and $X_{S-}=\underline X_1$, so that the minimum of $X$ is reached at least twice. 
The contradiction follows from negating \eqref{eqnXJumpsFromItsInfimum}, which proves its validity.

Conversely, assume $X$ jumps from its infimum with positive probability. 
Then equation \eqref{eqnXJumpsFromItsInfimum} holds true (though only with positive probability). 
Since $X$ is continuous at zero, then $\tau_0^+\in (0,1)$, which implies $X$ is in $(-\infty,0]$ on $(0,\tau_0^+)$.
This means $X$ is irregular upward with positive probability; being irregular upward is a tail event for the uniform random variables defining $X$, therefore, its probability is zero or one.

Using similar arguments we can prove $X$ is irregular downward if and only if $X$ jumps to its infimum.

Finally, Theorem \ref{RogozinForEITheorem} shows that $X$ is both regular upward and downward when it is of infinite variation, so that $X$ reaches its minimum continuously. 
\end{proof}

\subsection{EI processes conditioned to remain positive}

The aim of this subsection is to prove Theorem \ref{DIMTheorem}. 
As before, let $X$ be an extremal EI process with parameters $(0,\sigma,\beta)$. 
Assume that $X$ is both upward and downward regular. 
Since $\alpha=0$, either $X$ has infinite activity ($\sum_{i}\indi{\beta_i\neq 0}=\infty$) or a Gaussian component ($\sigma=0$). 
Otherwise, $X$ would have piecewise linear trajectories with the same slope $\sum_i \beta_i$; but then, $X$ would not be either upward or downward regular. 
Hence, $X$ satisfies \defin{UM}; let $\tau$ be the unique time $X$ reaches its minimum. 
Theorem \ref{DIMTheorem} tells us that $X$ is continuous at $\tau$. 
Corollary 3.1 in \cite{MR3558138} tells us that under these hypotheses, the law of $X$ conditioned to remain above $-\eps$ converges weakly to the Vervaat transform of $X$, given by $X_{\rho+\cdot}-X_{\rho}$. 
What was needed in the above cited corollary were conditions that would allow one to apply it 
and we have identified them in terms of regularity of both half-lines. 
In the particular case when $X$ is of infinite variation, 
Theorem \ref{RogozinForEITheorem} tells us that $X$ is both upward and downward regular 
and that therefore, the conclusion of Theorem \ref{DIMTheorem} is satisfied. 

The reader might wonder why we had to impose $\alpha=0$. 
The reference \cite{MR3558138} has a description of what could be the limit when $\alpha>0$ and $\beta=0$ 
(that is, for a brownian bridge from $0$ to $\alpha$). 
The candidate for a limit is described as a random shift, just as the Vervaat transformation for the case $\alpha=0$, 
but it needs a bicontinuous family of (non-zero!) local times in its definition. 
Defining such a process for an EI process is an open problem; 
semimartingale local times are only non-zero when $\sigma>0$, so a different approach is needed. 
Note that a limit theorem is not provided in \cite{MR3558138}. 

\subsection{The convex minorant of EI processes}

Let $X$ be an extremal EI process with parameters $(\alpha, \sigma, \beta)$ where, 
for concreteness,  $\beta_i\geq 0$ for all $i$. 

To prove Theorem \ref{ConMinTheorem}, we will rely strongly on \cite{MR2978134}. 

First, we establish some basic properties of the convex minorant in analogy with \cite[Proposition 1]{MR2978134}. 
They will be fundamental in applying a transformation in Skorohod space, 
which is continuous on paths satisfying the conclusion. 
\begin{proposition}
\label{proposition:BasicPropertiesConMin}
Assume that $X$ satisfies \defin{NPL} and 
let $C$ be the convex minorant of $X$. 
Then\begin{enumerate}
\item The open set $\mc{O}=\set{t\in [0,1]: C_{t}<X_{t}\wedge X_{t-}}$ has Lebesgue measure $1$. 
\item For every connected component $(g,d)$ of $\mc{O}$, $\Delta X_g\Delta X_d\geq 0$. 
If $X$ has infinite variation, $\Delta X_g\Delta X_d=0$. 
\item If $(g_1,d_1)$ and $(g_2,d_2)$ are connected components of $\mc{O}$, then
\[
\frac{C_{d_1}-C_{g_1}}{d_1-g_1}
\neq \frac{C_{d_2}-C_{g_2}}{d_2-g_2}. 
\]
\end{enumerate}
\end{proposition}

The proof of the above proposition is almost the same as the corresponding one in \cite{MR2978134}. 
We just need to apply different results. 
For example, the fact that when $X$ has finite variation, 
$\underline D(X)=\overline D(X)=\tilde \alpha$ (in the parametrization for this case), 
which is found in \cite[Prop. 4, p. 81]{MR1406564} for L\'evy processes, 
is now found in \cite[Cor. 3.30, p. 161]{MR2161313} for EI processes. 
(We have already used this result.) 
Or, D\"oblin's result that non-piecewise linear L\'evy processes have continuous distributions, 
has a counterpart for EI processes in \cite{MR1417982}, 
which also contains the fact that the minimum is reached 
in a unique place under \defin{NPL} (which implies \defin{UM}). 
One also needs our extensions of Millar's results stated in Theorem \ref{ZeroOneLawTheorem}, 
as well as the fact that
\[
\liminf_{t\to 0+}\frac{ X_{U_i+t}-X_{U_i}}{t}=-\infty,
\]at any jump time $U_i$ of $X$. 
This follows from \ref{RogozinForEITheorem} applied to $(X_t-\beta_i[\indi{U_i\leq t}-t])$. 

To prove Theorem \ref{ConMinTheorem}, we will use the following path transformation that leaves the laws of EI processes invariant. 

\begin{theorem}
\label{theorem:pathTransformationForConmin}
Let $X$ be an extremal EI process of parameters $(\alpha,\sigma,\beta)$ satisfying \defin{NPL}. 
Define its convex minorant $C$ and the open set of excursion intervals $\mc{O}$ as before. 
Let $U$ be an uniform random variable on $(0,1)$ 
independent of $X$ and consider the connected component $(d,g) $ of $\mc{O}$ that contains $U$. 
Define the 3214 transformation $X^U$ of $X$ by means of
\[
X^U_t=\begin{cases}
X_{U+t}-X_U, & 0\leq t<d-U\\
C_d-C_g+ X_{g+t-(d+U)}-X_U& d-U\leq t\leq d-g\\
C_d-C_g+X_{t-(d-g)}& d-g\leq t<d\\
X_t& d\leq t\leq 1
\end{cases}. 
\]Then, $(U,X)\stackrel{d}{=}(d-g,X^U)$. 
\end{theorem}
Remark that $U$ belongs almost surely to $\mc{O}$, 
since the latter has Lebesgue measure $1$ by Proposition \ref{proposition:BasicPropertiesConMin}. 

The above path transformation can be understood as follows: 
the random variable $U$ is used to select a face of the convex minorant of $X$, with endpoints $g$ and $d$. 
This divides the trajectory into 4 parts, say $1,2,3$ and $4$ which are then rearranged as $3,2,1,4$. 
Parts $1$ and $4$ have the same convex minorant as $X$, with the selected face removed. 
Parts $3$ and $2$ are interpreted as an inverse Vervaat transformation; 
the original trajectory $2$ and $3$ can be obtained 
as the Vervaat transform of the Knight bridge of $X^U$ on $[0,d-g]$. 
Once of the consequences is that $d-g$ has the same law as $U$, 
which is a remarkable universality result for exchangeable increment processes 
and is responsible for the stick breaking process of Theorem \ref{ConMinTheorem}. 
Indeed, we just need to iterate the path transformation on parts $1$ and $4$ of the trajectory of $X^U$. 
Therefore, Theorem \ref{ConMinTheorem} follows from Theorem \ref{theorem:pathTransformationForConmin}, 
whose proof we now sketch, being very similar to the proof for L\'evy processes of \cite{MR2978134}. 
It is based on an analogous path transformation for discrete time EI processes stated in \cite[Theorem 8.1]{MR2825583} or \cite[Lemma 7]{MR2831081}; the proof of the latter is by means of a bijection between permutations.   
To pass to the limit, one uses the continuity of the path transformation, on Skorohod space, whenever the trajectory satisfies the basic properties of Proposition \ref{proposition:BasicPropertiesConMin}, see Section 6.3 of \cite{MR2978134}. 
Continuity of the path transformation is mucho more simple when $X$ is of infinite variation since then $X$ is continuous at $g$ and $d$. See Section 6.2 of \cite{MR2978134}.

We end the paper with an explanation of the distributional description of the maximum (or minimum, after multiplication by $-1$) of an EI process, which in discrete time is  displayed in Equation \eqref{equation:distributionalEqualityForMaximum}, and how it proves the celebrated  formula due to M. Kac, which in discrete time is Equation \eqref{equation:SpitzerExpectationOfMaximum}. 
Indeed, note that the infimum $\underline X_1$ of $X$ on $[0,1]$ is the sum of the increments of the convex minorant that are negative. 
Thanks to Theorem \ref{ConMinTheorem}, this gives us the equality in law
\[
\underline X_1\stackrel{d}{=}\sum_{i} \bra{X_{S_i}-X_{S_{i-1}}}^-. 
\]Next, conditioning on the stick-breaking process $L$, we see that
\[
\esp{\underline X_1}=\esp{\sum_{i} \bra{X_{S_i}-X_{S_{i-1}}}^-}=\sum_i \esp{\imf{f}{S_{i-1},S_i}},
\]where $\imf{f}{r,s}=\esp{[X_{s}-X_r]^{-}}$ for $r<s$. 
However, exchangeability implies that $f(r,s)=f(0,s-r)$, so that
\[
\esp{\underline X_1}=\sum_i  \esp{\imf{g}{L_i}}
\]where $g(l)=f(0,l)$. Finally, recall that the uniform stick-breaking process is invariant under size-biased permutations. 
Indeed, it is itself a size-biased permutation of a non-decreasing sequence; cf. \cite[Corollaries 7 and 8]{MR1337249} or the following comment from \cite[p. 57]{MR2245368}: \emph{[The uniform stick-breaking process] has the same distribution as the size-biased permutation of the jumps of the Dirichlet process[...]. }
In particular, if conditionally on $L$, the index $I$ has the law
\[
\probac{I=i}{L}=L_i,
\]then $L_1\stackrel{d}{=} L_I$. 
Hence, we obtain that for any  $\fun{h}{[0,1]}{\re_-}$, 
\[
\esp{\sum_i h(L_i)L_i}=\esp{h(L_I)}=\esp{h(L_1)}=\int_0^1 h(s)\, ds. 
\]Applying the above result to $h(l)=g(l)/l$ gives
\[
\esp{\underline X_1}=\int_0^1 \frac{\esp{X_l^-}}{l}\, dl. 
\]

\bibliography{GenBib}
\bibliographystyle{amsalpha}
\end{document}

%% file: definitions.tex
\DeclareMathOperator{\id}{Id} %
\newcommand{\defin}[1]{{\bf #1}}
\newcommand{\mc}[1]{\ensuremath{\mathscr{#1}}}
\newcommand{\bb}[1]{\mathbb{#1}}

\newcommand{\fun}[3]{\ensuremath{#1:#2\to #3}}

\newcommand{\set}[1]{\ensuremath{\left\{ #1\right\} }}

\newcommand{\paren}[1]{\ensuremath{\left( #1\right) }}
\newcommand{\bra}[1]{\ensuremath{\left[ #1\right] }}

\newcommand{\sip}{\bb{P}}

\newcommand{\cadlag}{c\`adl\`ag}
\newcommand{\se}{\ensuremath{\bb{E}}}
\newcommand{\ssa}{\ensuremath{\mathscr{F}}}
\newcommand{\si}{{\ensuremath{\bf{1}}}}

\newcommand{\sa}{\ensuremath{\sigma}\nbd field}

\newcommand{\eps}{\ensuremath{ \varepsilon}}
\newcommand{\na}{\ensuremath{\mathbb{N}}}

\newcommand{\re}{\ensuremath{\mathbb{R}}}

\newcommand{\proba}[1]{\ensuremath{\sip\! \left( #1 \right)}}

\newcommand{\probac}[2]{\ensuremath{\sip\! \left( #1 \, | #2 \right)}}
\newcommand{\esp}[1]{\ensuremath{\se\! \left( #1 \right)}}

\newcommand{\var}[1]{\ensuremath{\text{Var}\! \left( #1 \right)}}

\newcommand{\abs}[1]{\hspace{.25mm}\left|#1\right|\hspace{.25mm}}

\newcommand{\F}{\ssa}

\newcommand{\G}{\ensuremath{\mc{G}}}

\newcommand{\indi}[1]{\si_{#1}}

\newcommand{\p}{\ensuremath{ \sip  } }
\newcommand{\q}{\ensuremath{ \bb{Q}  } }
\newcommand{\ofp}{\ensuremath{ \paren{ \Omega ,\F ,\p } } }

\newcommand{\sag}[1]{\sigma\!\paren{#1}}

\newcommand{\nbd}{\nobreakdash -}

\newcommand{\imf}[2]{\ensuremath{#1\!\paren{#2}}}

\newenvironment{esn}{\begin{equation*}}{\end{equation*}}

\renewcommand{\sa}{$\sigma$\nbd \'algebra}

\usepackage[dvipsnames]{xcolor}